\documentclass[12pt, dvipdfmx]{article}
\usepackage[mathscr]{eucal}
\usepackage{amssymb}
\usepackage{latexsym}
\usepackage{amsthm}
\usepackage{amsmath}
\usepackage[dvips]{graphicx}
\usepackage{psfrag}
\usepackage{a4wide}
\usepackage{tikz-cd}

\newtheorem{defn}{Definition}[section]
\newtheorem{thm}{Theorem}[section]

\newtheorem{lem}{Lemma}[section]
\newtheorem{rem}{Remark}[section]
\newtheorem{ex}{Example}[section]
\newtheorem{cor}{Corollary}[section]
\newtheorem{prop}{Proposition}[section]

\numberwithin{equation}{section}
\newcommand{\C}{{\mathbb C}}

\newcommand{\fT}{{\mathfrak T}}
\newcommand{\fM}{{\mathfrak M}}
\newcommand{\cM}{{\mathcal M}}
\newcommand{\cL}{{\mathcal L}}

\newcommand{\R}{{\mathbb R}}

\newcommand{\bT}{{\mathbb T}}
\newcommand{\cH}{{\mathcal H}}
\newcommand{\cK}{{\mathcal K}}
\newcommand{\cG}{{\mathcal G}}
\newcommand{\ran}{\operatorname{ran}}

\newcommand{\la}{\langle}
\newcommand{\ra}{\rangle}

\newcommand{\lam}{\lambda}

\begin{document}

\title{An indefinite range inclusion theorem\\ for triplets of bounded linear operators\\ on a Hilbert space
\thanks{The research was supported by JSPS KAKENHI Grant Number 15K04926.}
}
\author{
{\sc Michio SETO}\\
[1ex]
{\small National Defense Academy,  
Yokosuka 239-8686, Japan} \\
{\small 
{\it E-mail address}: {\tt mseto@nda.ac.jp}}\\
and\\
{\sc Atsushi UCHIYAMA}\\
[1ex]
{\small Yamagata University, Yamagata 990-8560, Japan} \\
{\small 
{\it E-mail address}: {\tt uchiyama@sci.kj.yamagata-u.ac.jp}}
}

\date{}

\maketitle
\begin{abstract}
We study triplets of Hilbert space operators satisfying a certain inequality.   
A range inclusion theorem with norm estimate for those triplets is given 
with the language of Kre\u{\i}n space geometry and 
de Branges-Rovnyak space theory. 
\end{abstract}

\begin{center}
keywords: Kre\u{\i}n space, de Branges-Rovnyak space, Toeplitz operator
\end{center}

\section{Introduction}
Let $T_1$, $T_2$ and $T_3$ be bounded linear operators on a Hilbert space $\cH$. 
In this paper, we are going to study triplet $(T_1,T_2,T_3)$ satisfying the following inequality: 
\begin{equation}\label{eq:1-1}
0\leq T_1T_1^{\ast}+T_2T_2^{\ast}-T_3T_3^{\ast}\leq I.
\end{equation}

Let $\fT(\cH)$ denote the set of operator triplets satisfying (\ref{eq:1-1}) on $\cH$. 
For any triplet $(T_1,T_2,T_3)$ in $\fT(\cH)$, we set 
\[
T=(T_1T_1^{\ast}+T_2T_2^{\ast}-T_3T_3^{\ast})^{1/2}.
\]

As the main theorem of this paper, we will show the following: 
for any vector $u$ in $\ran T$, there exists some vector 
$\mathbf{z}_{\varepsilon}=(z_1(\varepsilon),z_2(\varepsilon),z_3(\varepsilon))^t$ 
in $\cH\oplus \cH \oplus \cH$ such that 
\begin{enumerate}
\item $T_1z_1(\varepsilon)+T_2z_2(\varepsilon)-T_3z_3(\varepsilon)\to u$ 
$(\varepsilon\to 0)$ in the strong topology of $\cH$,
\item
$0\leq \|z_1(\varepsilon)\|_{\cH}^2+\|z_2(\varepsilon)\|_{\cH}^2-\|z_3(\varepsilon)\|_{\cH}^2
\uparrow  \|u\|_{\cM(T)}^2$ $(\varepsilon\downarrow 0)$,
\end{enumerate}
where $\|u\|_{\cM(T)}$ denotes the norm of $u$ in the de Branges-Rovnyak space induced by $T$. 

This paper is organized as follows. 
In Section 2, 
by giving examples from operator theory on Hardy spaces, 
it is shown that $\fT(\cH)$ is nontrivial. 
In Section 3, 
we study indefinite inner product spaces induced by triplets in $\fT(\cH)$, and 
prove the main theorem (Theorem \ref{thm:4-1}). 
In Section 4, we investigate the local structure of range spaces of operators appearing in Section 3. 

\section{Examples}
Trivial examples of triplets in $\fT(\cH)$ are easily obtained from Douglas' range inclusion theorem.  
We shall see that $\fT(\cH)$ is nontrivial. 

\begin{ex}\label{ex:1-1}\rm 
Let $H^2$ be the Hardy space over the unit disk, 
and let $H^{\infty}$ be the Banach algebra consisting of all bounded analytic functions in $H^2$. 
For any function $\varphi$ in  $H^{\infty}$, 
$T_{\varphi}$ denotes the Toeplitz operator with symbol $\varphi$. 
We choose $\varphi_1$ and $\varphi_2$ from $H^{\infty}$ satisfying 
\[
\| \begin{pmatrix}
T_{\varphi_1} & T_{\varphi_2}
\end{pmatrix}
 \| \leq 1.
\]
Then this norm inequality is equivalent to that 
\[
0\leq T_{\varphi_1}T_{\varphi_1}^{\ast}+T_{\varphi_2}T_{\varphi_2}^{\ast}\leq I. 
\]
Further, 
we choose $\psi_1$ and $\psi_2$ from $H^{\infty}$ satisfying 
\[
\| \begin{pmatrix}
T_{\psi_1} \\ 
T_{\psi_2}
\end{pmatrix}
 \| \leq 1.
\]
Then, setting  
\[
\varphi_3=\varphi_1\psi_1+\varphi_2\psi_2=
\begin{pmatrix}
\varphi_1 & \varphi_2
\end{pmatrix}
\begin{pmatrix}
\psi_1 \\ 
\psi_2
\end{pmatrix},
\] 
by the generalized Toeplitz-corona theorem (see Theorem 8.57 in Agler-McCarthy~\cite{AM}), 
we have that
\[
0
\leq T_{\varphi_1}T_{\varphi_1}^{\ast}+T_{\varphi_2}T_{\varphi_2}^{\ast}-T_{\varphi_3}T_{\varphi_3}^{\ast}
\leq T_{\varphi_1}T_{\varphi_1}^{\ast}+T_{\varphi_2}T_{\varphi_2}^{\ast}
\leq I.   
\]
\end{ex}

Our study has been motivated by the next example.  
\begin{ex}[Wu-Seto-Yang~\cite{WSY}]\label{ex:2-2}\rm 
Let $H^2$ be the Hardy space over the unit disk. 
Then the tensor product Hilbert space $H^2 \otimes H^2$ 
is isomorphic to the Hardy space over the bidisk. 
Let $z$ and $w$ denote coordinate functions, 
and let $T_z$ and $T_w$ be Toeplitz operators with symbols $z$ and $w$, respectively.   
We note that $T_z$ and $T_w$ are doubly commuting isometries on $H^2\otimes H^2$. 
In fact, 
$T_z$ and $T_w$ are identified with $T_z\otimes I$ and $I\otimes T_w$, respectively. 
Now,  since orthogonal projections $T_zT_z^{\ast}$ and $T_w(I-T_zT_z^{\ast})T_w^{\ast}$ 
are commuting, 
\[
T_zT_z^{\ast}+T_wT_w^{\ast}-T_{zw}T_{zw}^{\ast}
=T_zT_z^{\ast}+T_w(I-T_zT_z^{\ast})T_w^{\ast},
\]
is the orthogonal projection onto $(H^2\otimes H^2)\ominus \C$. 
Hence $(T_z,T_w,T_{zw})$ belongs to $\fT(H^2\otimes H^2)$. 
Further non-trivial examples can be obtained from the module structure of $H^2\otimes H^2$.  
Let $\cM$ be a closed subspace of $H^2\otimes H^2$. 
Then $\cM$ is called a submodule if $\cM$ is invariant for $T_z$ and $T_w$. 
For many examples of submodules in $H^2\otimes H^2$, 
there exist bounded analytic functions $\varphi_1$, $\varphi_2$ and $\varphi_3$ on the bidisk 
such that
\[
T_{\varphi_1}T_{\varphi_1}^{\ast}+T_{\varphi_2}T_{\varphi_2}^{\ast}-T_{\varphi_3}T_{\varphi_3}^{\ast}
=P_{\cM},
\]
where $P_{\cM}$ denotes the orthogonal projection onto $\cM$, 
and 
\[
T_{\varphi_1}^{\ast}T_{\varphi_1}+T_{\varphi_2}^{\ast}T_{\varphi_2}-T_{\varphi_3}^{\ast}T_{\varphi_3}
=I.
\] 
\end{ex}

\begin{rem}\rm 
We note that Example \ref{ex:1-1} is deduced from the complete Pick property. 
On the other hand, the kernel of $H^2\otimes H^2$ does not have it. 
\end{rem}

\section{Indefinite range inclusion}
Setting
\[\cH_+=\cH\oplus \cH,\quad \cH_-=\cH \quad \mbox{and}\quad   
J=
\begin{pmatrix}
1 & 0 & 0\\
0 & 1 & 0\\
0 & 0& -1
\end{pmatrix},
\]
we consider the Kre\u{\i}n space $\cK=(\cH_+\oplus \cH_-,J)$, that is, 
for any vectors $\mathbf{x}=(x_1,x_2,x_3)^t$ and $\mathbf{y}=(y_1,y_2,y_3)^t$ in $\cH\oplus \cH \oplus \cH$, 
the inner product of $\mathbf{x}$ and $\mathbf{y}$ in $\cK$ is defined to be 
\[
\la \mathbf{x},\mathbf{y} \ra_{\cK}=\la J\mathbf{x},\mathbf{y} \ra_{\cH\oplus\cH\oplus\cH}= 
\la x_1,y_1\ra_{\cH}+\la x_2,y_2\ra_{\cH}-\la x_3, y_3\ra_{\cH}.
\]
For basic Kre\u{\i}n space geometry, see Dritschel-Rovnyak~\cite{DR}. 

Let $(T_1,T_2,T_3)$ be a triplet in $\fT(\cH)$.  
Then we define a linear operator $\bT$ as follows:
\[\bT: \cK \to \cH,\quad 
\begin{pmatrix}
x_1\\
x_2\\
x_3
\end{pmatrix}
\mapsto T_1x_1+T_2x_2-T_3x_3.
\]
The adjoint operator $\bT^{\sharp}$ of $\bT$ with respect to inner products of $\cH$ and $\cK$ is obtained as follows: 
\begin{align*}
\la x,\bT (x_1,x_2,x_3)^t \ra_{\cH}
&=\la x,T_1x_1+T_2x_2-T_3x_3 \ra_{\cH}\\
&=\la T_1^{\ast}x,x_1 \ra_{\cH}+\la T_2^{\ast}x,x_2 \ra_{\cH}
-\la T_3^{\ast}x,x_3 \ra_{\cH}\\
&=\la (T_1^{\ast}x,T_2^{\ast}x,T_3^{\ast}x)^t,(x_1,x_2,x_3)^t \ra_{\cK},
\end{align*}
that is, we have that
\[\bT^{\sharp}:\cH\to \cK, \quad x\mapsto
\begin{pmatrix}
T_1^{\ast}x\\
T_2^{\ast}x\\
T_3^{\ast}x
\end{pmatrix}.
\]
In particular, we have that 
\[ 
\bT\bT^{\sharp} x
=T_1T_1^{\ast}x+T_2T_2^{\ast}x-T_3T_3^{\ast}x.
\]
For any $(T_1,T_2,T_3)$ in $\fT(\cH)$,   
we set 
\[T=(T_1T_1^{\ast}+T_2T_2^{\ast}-T_3T_3^{\ast})^{1/2}.
\] 
Note that $T$ is positive and contractive. 
Consider the operator $V: \ran T\to \bT^{\sharp}(\ker T)^{\perp}$ defined by 
\[VTx=
\begin{pmatrix}
T_1^{\ast}x\\
T_2^{\ast}x\\
T_3^{\ast}x
\end{pmatrix}
\quad (x\in (\ker T)^{\perp}). 
\]
Then 
it follows from the identity 
\begin{equation}\label{eq:3-0}
\| Tx \|_{\cH}^2=
\|T_1^{\ast} x\|_{\cH}^2+\|T_2^{\ast} x\|_{\cH}^2-\|T_3^{\ast} x\|_{\cH}^2=\la \bT^{\sharp}x, \bT^{\sharp}x\ra_{\cK}
\end{equation}
that $V$ is an isometry and 
$\bT^{\sharp}(\ker T)^{\perp}$ is a pre-Hilbert space. 
Let $\cK_0$ be the completion of $\bT^{\sharp}(\ker T)^{\perp}$ 
with the norm induced by (\ref{eq:3-0}), 
and let $\widetilde{V}: \overline{\ran}T\to \cK_0$ denote the isometric extension of $V$, 
in fact, $\widetilde{V}$ is unitary.  
Then $\bT^{\sharp}=\widetilde{V}T$ on $(\ker T)^{\perp}$ 
gives the polar decomposition of $\bT^{\sharp}$, that is,
\begin{equation}\label{eq:3-0-1}
\begin{tikzcd}
\cL \arrow[r,"\bT^{\sharp}"]\arrow[d,"T",swap] & \cK_0\\
\cL \arrow[ru,"\widetilde{V}",swap] & 
\end{tikzcd}
\quad (\cL:=(\ker T)^{\perp}=\overline{\ran}T)
\end{equation}
is commutative. 
 
Further, it follows from (\ref{eq:3-0}) that $\bT^{\sharp}$ is bounded in (\ref{eq:3-0-1}). 
Hence, we can take the Hilbert space adjoint $\widetilde{\bT}$ 
of $\bT^{\sharp}$ in (\ref{eq:3-0-1}). 
We summarize basic properties of $\widetilde{\bT}$ in the following proposition: 
\begin{prop}
Let $\widetilde{\bT}:\cK_0\to \cL$ 
be the Hilbert space adjoint of $\bT^{\sharp}$ in the sense of {\rm(\ref{eq:3-0-1})}. 
Then
\begin{enumerate}
\item[\rm (i)] $\widetilde{\bT}$ is injective, 
\item[\rm (ii)] $\widetilde{\bT}$ is the extension of $\bT|_{\bT^{\sharp}\cL}$,
\item[\rm (iii)] $\ran \widetilde{\bT}$ is dense in $\cL$.
\end{enumerate}
\end{prop}

\begin{proof}
Suppose that $\widetilde{\bT}\mathbf{x}=0$ for some $\mathbf{x}$ in $\cK_0$. 
Then there exists a sequence $\{w_n\}_n$ in $\cL=(\ker T)^{\perp}$ such that 
$\|\mathbf{x}-\bT^{\sharp} w_n\|_{\cK_0}\to 0$ as $n\to \infty$. 
Hence we have that 
\[
\|\mathbf{x}\|_{\cK_0}^2
=\la \mathbf{x},\mathbf{x} \ra_{\cK_0}
=\lim_{n\to \infty}\la \mathbf{x},\bT^{\sharp}w_n \ra_{\cK_0}
=\lim_{n\to \infty}\la \widetilde{\bT}\mathbf{x},w_n \ra_{\cH}=0.
\]
Thus we have (i). 
Further, since 
\[
\la \bT\bT^{\sharp}x,y \ra_{\cH}
=\la \bT^{\sharp}x,\bT^{\sharp}y \ra_{\cK}
=\la \bT^{\sharp}x,\bT^{\sharp}y \ra_{\cK_0}
=\la \widetilde{\bT}\bT^{\sharp}x,y \ra_{\cH}\quad (x,y \in \cL),
\]
we have (ii). It follows from (ii) that 
\[
\cL \supset 
\widetilde{\bT}\cK_0\supset 
\bT \bT^{\sharp}\cL=\ran T^2.
\]
This concludes (iii). 
\end{proof}

Let $\cM(T)$ denote the de Branges-Rovnyak space induced by $T$, that is, 
$\cM(T)$ is the Hilbert space consisting of all vectors in $\ran T$ 
with the pull-back norm 
\[
\|Tx\|_{\cM(T)}
=\|P_{(\ker T)^{\perp}}x\|_{\cH}
=\min\{ \|y\|_{\cH}:Ty=Tx \},
\] 
where $P_{(\ker T)^{\perp}}$ is the orthogonal projection from $\cH$ onto $(\ker T)^{\perp}$.

\begin{thm}\label{thm:4-1}
Let $\cH$ be a Hilbert space. 
For any triplet $(T_1,T_2,T_3)$ in $\fT(\cH)$, we set 
\[
T=(T_1T_1^{\ast}+T_2T_2^{\ast}-T_3T_3^{\ast})^{1/2}.
\]
If $u$ belongs to $\cM(T)$, then, for any $\varepsilon>0$, there exists some vector  
$\mathbf{z}_{\varepsilon}=(z_1(\varepsilon),z_2(\varepsilon),z_3(\varepsilon))^t$ 
in $\cH\oplus \cH \oplus \cH$ such that 
\begin{enumerate}
\item[\rm (i)] $T_1z_1(\varepsilon)+T_2z_2(\varepsilon)-T_3z_3(\varepsilon)\to u$ 
$(\varepsilon\to 0)$ in the strong topology of $\cH$,
\item[\rm (ii)]
$0\leq \|z_1(\varepsilon)\|_{\cH}^2+\|z_2(\varepsilon)\|_{\cH}^2-\|z_3(\varepsilon)\|_{\cH}^2
\uparrow  \|u\|_{\cM(T)}^2$ $(\varepsilon\downarrow 0)$,
\item[\rm (iii)] 
$\mathbf{z}_{\varepsilon}$ converges to some vector $\mathbf{z}$ in the strong topology of $\cK_0$ 
and $u=\widetilde{\bT}\mathbf{z}$. 
\end{enumerate}
\end{thm}

\begin{proof}
Let $\{E_{\lam}\}_{\lam\in \R}$ be the spectral family of $T$, 
and we set 
\[E_{\lam}^{\perp}=I_{\cH}-E_{\lam}=E( (\lam,\infty) ).
\]
Suppose that $u=Tx$ where $x$ is in $\cL=(\ker T)^{\perp}$. 
Then, for arbitrary $\varepsilon>0$, 
put 
\[
x_{\varepsilon}=E_{\varepsilon}^{\perp}x,\quad 
y_{\varepsilon}=\left(\int_{\varepsilon}^{\infty}\frac{1}{\lam}\ dE_{\lam}\right)x_{\varepsilon}
\quad \mbox{and}\quad
\mathbf{z}_{\varepsilon}=\bT^{\sharp}y_{\varepsilon}.
\]
We note that $y_{\varepsilon}$ and $\mathbf{z}_{\varepsilon}$ belong 
to $\cL$ and $\cK_0$, respectively.  
Then we have that
\[
\bT \mathbf{z}_{\varepsilon}=\bT\bT^{\sharp}y_{\varepsilon}
=T^2y_{\varepsilon}=T^2\left(\int_{\varepsilon}^{\infty}\frac{1}{\lam}\ dE_{\lam}\right)x_{\varepsilon}
=Tx_{\varepsilon}.
\]
Hence we have that
\[
\|u-\bT \mathbf{z}_{\varepsilon}\|_{\cH}=\|Tx-Tx_{\varepsilon}\|_{\cH}
\leq  \|x-x_{\varepsilon}\|_{\cH}\to 0\quad (\varepsilon\to 0).
\]
Thus we have (i). 
Further, 
it follows from $Ty_{\varepsilon}=x_{\varepsilon}$ that 
\begin{align*}
\|z_1(\varepsilon)\|_{\cH}^2+\|z_2(\varepsilon)\|_{\cH}^2-\|z_3(\varepsilon)\|_{\cH}^2
&=\|T_1^{\ast}y_{\varepsilon}\|_{\cH}^2
+\|T_2^{\ast}y_{\varepsilon}\|_{\cH}^2-\|T_3^{\ast}y_{\varepsilon}\|_{\cH}^2\\
&=\|Ty_{\varepsilon}\|_{\cH}^2\\
&=\|x_{\varepsilon}\|_{\cH}^2\\
&\leq \|x\|_{\cH}^2.
\end{align*}
This concludes (ii). Finally, since 
\begin{align*}
\|\mathbf{z}_{\varepsilon}-\mathbf{z}_{\delta}\|_{\cK_0}^2
&=\|\bT^{\sharp}y_{\varepsilon}-\bT^{\sharp}y_{\delta}\|_{\cK_0}\\
&=\|T_1^{\ast}(y_{\varepsilon}-y_{\delta})\|_{\cH}^2
+\|T_2^{\ast}(y_{\varepsilon}-y_{\delta})\|_{\cH}^2
-\|T_3^{\ast}(y_{\varepsilon}-y_{\delta})\|_{\cH}^2\\
&=\|T (y_{\varepsilon}-y_{\delta})\|_{\cH}^2\\
&=\|x_{\varepsilon}-x_{\delta}\|_{\cH}^2\\
&\to 0\quad (\varepsilon,\delta \to 0),
\end{align*}
$\mathbf{z}_{\varepsilon}$ converges to some vector $\mathbf{z}$ in $\cK_0$, and 
\[
u
=\lim_{\varepsilon\to 0}\bT \mathbf{z}_{\varepsilon}
=\lim_{\varepsilon\to 0}\widetilde{\bT} \mathbf{z}_{\varepsilon}
=\widetilde{\bT}\mathbf{z}.
\]
Thus we have (iii). 
\end{proof}

\begin{cor}\label{cor:3-1}
Suppose that $\cH$ is of finite dimension. 
If $u$ belongs to $\cM(T)$, then there exists some 
$\mathbf{z}=(z_1,z_2,z_3)^t$ in $\cH\oplus \cH \oplus \cH$ such that 
\[
T_1z_1+T_2z_2-T_3z_3= u
\]
and
\[
\|z_1\|_{\cH}^2+\|z_2\|_{\cH}^2-\|z_3\|_{\cH}^2
= \|u\|_{\cM(T)}^2.
\]
\end{cor}

\begin{proof}
If $\varepsilon>0$ is sufficiently small,
then $E_{\varepsilon}^{\perp}=P_{(\ker T)^{\perp}}$. 
\end{proof}

Theorem \ref{thm:4-1} can be generalized 
to any finite operator tuple $(T_1,\ldots,T_m,T_{m+1},\ldots, T_n)$ 
satisfying
\[
\sum_{j=1}^mT_jT_j^{\ast}-\sum_{k=m+1}^nT_kT_k^{\ast}\geq 0.
\] 

In particular, 
the proof of Theorem \ref{thm:4-1} 
can be applied to 
the de Branges-Rovnyak complement $\cH(T)=\cM(\sqrt{I-TT^{\ast}})$ of $\cM(T)$. 

\begin{thm}\rm
Let $\cH$ be a Hilbert space. 
For any triplet $(T_1,T_2,T_3)$ in $\fT(\cH)$, we set 
\[
T=(T_1T_1^{\ast}+T_2T_2^{\ast}-T_3T_3^{\ast})^{1/2}.
\]
Then, for any $v$ in $\cH(T)$ and $\varepsilon>0$, 
there exists some vector  
$\mathbf{w}_{\varepsilon}=(w_0(\varepsilon),w_1(\varepsilon),w_2(\varepsilon),w_3(\varepsilon))^t$ 
in $\cH\oplus \cH \oplus \cH\oplus \cH$ such that 
\begin{enumerate}
\item[\rm (i)] $w_0(\varepsilon)-T_1w_1(\varepsilon)-T_2w_2(\varepsilon)+T_3w_3(\varepsilon)\to v$ 
$(\varepsilon\to 0)$ in the strong topology of $\cH$,
\item[\rm (ii)]
$0\leq \|w_0(\varepsilon)\|_{\cH}^2-\|w_1(\varepsilon)\|_{\cH}^2
-\|w_2(\varepsilon)\|_{\cH}^2+\|w_3(\varepsilon)\|_{\cH}^2
\uparrow  \|v\|_{\cH(T)}^2$ $(\varepsilon\downarrow 0)$.
\end{enumerate}
\end{thm}

In the proof of Theorem \ref{thm:4-1}, 
we essentially showed that 
$\cM(T)$ is contractively embedded into $\cM(\widetilde{\bT})$. 
Moreover, applying the same method in the proof of Theorem \ref{thm:3-4} stated in the next section, 
we can conclude that the converse is also true, that is, 
$\cM(T)=\cM(\widetilde{\bT})$ as Hilbert spaces. 
However, $\cK_0$ and $\widetilde{\bT}$ seem to be rather elusive obejects. 
Thus, in the next section, we will investigate the local structure of range spaces of $T$ and $\bT$.

\section{Local structure of range spaces}

In this section, we need some facts from de Branges-Rovnyak space theory and Kre\u{\i}n space geometry. 
Let $\cH$ and $\cG$ be Hilbert spaces, 
and let $A$ be any bounded linear operator from $\cH$ to $\cG$.   
Then $\cM(A)$ denotes the de Branges-Rovnyak space induced by $A$, that is, 
$\cM(A)$ is the Hilbert space consisting of all vectors in $\ran A$ 
with the pull-back norm 
\[
\|Ax\|_{\cM(A)}
=\|P_{(\ker A)^{\perp}}x\|_{\cH}
=\min\{ \|y\|_{\cH}:Ay=Ax \},
\] 
where $P_{(\ker A)^{\perp}}$ denotes the orthogonal projection onto $(\ker A)^{\perp}$. 

The following theorem seems to be well known to specialists in Hilbert space operator theory. 
\begin{thm}\label{thm:3-3}
Let $A$ be a bounded linear operator from $\cH$ to $\cG$ and 
let $u$ be a vector in $\cG$. 
Then $u$ belongs to $\ran A$ if and only if 
\[
\gamma:=\sup_{A^{\ast}y\neq0}\frac{|\la y,u \ra_{\cG}|}{\|A^{\ast}y\|_{\cH}}
\]
is finite. Further, then $\|u\|_{\cM(A)}=\gamma$. 
\end{thm}

The first half of Theorem \ref{thm:3-3} is 
known as Shmuly'an's theorem 
(see Corollary 2 of Theorem 2.1 in Fillmore-Williams~\cite{FW}).  
For the second half (norm identity) and also the proof, 
we referred to Ando~\cite{Ando}.

\begin{defn}\label{defn:3-1}
Let $\cK$ be a Kre\u{\i}n space. 
A subspace $\fM$ of $\cK$ is said to be uniformly positive with lower bound $\delta>0$
if 
\[
\la \mathbf{x},\mathbf{x}\ra_{\cK} \geq \delta \la J\mathbf{x},\mathbf{x} \ra_{\cK}
\quad (\mathbf{x}\in \fM). 
\] 
In particular, for the Kre\u{\i}n space defined in Section 3, 
if $\fM$ is uniformly positive with lower bound $\delta>0$, then  
\[
\|x_1\|_{\cH}^2+\|x_2\|_{\cH}^2-\|x_3\|_{\cH}^2 \geq \delta (\|x_1\|_{\cH}^2+\|x_2\|_{\cH}^2+\|x_3\|_{\cH}^2) 
\quad ((x_1,x_2,x_3)^t\in \fM).
\] 
\end{defn}

\begin{thm}\label{thm:3-1}
Let $\cK$ be a Kre\u{\i}n space. 
Every uniformly positive subspace of $\cK$ with lower bound $\delta>0$
 is contained in a maximal uniformly positive subspace with lower bound $\delta>0$, 
and every maximal uniformly positive subspace is a Hilbert space with the inner product of $\cK$. 
\end{thm}
For the details of Definition \ref{defn:3-1} and Theorem \ref{thm:3-1}, see Dritschel-Rovnyak~\cite{DR}. 

\begin{lem}\label{lem:3-1}
Let $(T_1,T_2,T_3)$ be any triplet in $\fT(\cH)$.  
Then, for any $\varepsilon>0$, 
\[
\fM_{\varepsilon}=\{(T_1^{\ast}x,T_2^{\ast}x,T_3^{\ast}x)^t: x \in E_{\varepsilon}^{\perp}\cH\}
\] 
is uniformly positive. 
\end{lem}

\begin{proof}
For any $\varepsilon >0$ and any vector $x$ in $E_{\varepsilon}^{\perp}\cH$, 
\begin{align*}
\|T_1^{\ast} x\|_{\cH}^2+\|T_2^{\ast} x\|_{\cH}^2-\|T_3^{\ast} x\|_{\cH}^2
&= \| Tx \|_{\cH}^2\\ 
&=\int_{\varepsilon}^{\infty}|\lam|^2\ d\|E_{\lam}x\|_{\cH}^2\\
&\geq \varepsilon^2 \|x\|_{\cH}^2\\
&\geq \frac{\varepsilon^2}{3\max_{1\leq j\leq 3} \|T_j^{\ast}\|^2}
(\|T_1^{\ast}x\|_{\cH}^2+\|T_2^{\ast}x\|_{\cH}^2+\|T_3^{\ast}x\|_{\cH}^2).
\end{align*}
This concludes the proof. 
\end{proof}

Let $\widetilde{\fM_{\varepsilon}}$ be 
a maximal uniformly positive subspace containing $\fM_\varepsilon$. 
Then, $\widetilde{\fM_{\varepsilon}}$ is a Hilbert space with the inner product of $\cK$
by Theorem \ref{thm:3-1},  
and hence so is $\overline{\fM_{\varepsilon}}$, the closure of $\fM_{\varepsilon}$ in $\widetilde{\fM_{\varepsilon}}$. 
We note that  $\overline{\fM_{\varepsilon}}$ is uniformly positive, that is, 
there exists some $\delta>0$ such that the following inequality holds:
\begin{equation}\label{eq:3-1}
\|x\|_{\cH}^2+\|y\|_{\cH}^2-\|z\|_{\cH}^2\geq \delta (\|x\|_{\cH}^2+\|y\|_{\cH}^2+\|z\|_{\cH}^2)\quad 
((x,y,z)^t\in \overline{\fM_{\varepsilon}}).
\end{equation}

\begin{lem}\label{thm:3-2}
Let $(T_1,T_2,T_3)$ be any triplet in $\fT(\cH)$.  
Then, for any $\varepsilon>0$, 
$\bT|_{\overline{\fM_{\varepsilon}}}: \overline{\fM_{\varepsilon}}\to \cH$ is bounded as a Hilbert space operator. 
\end{lem}

\begin{proof}\rm 
Since $\bT|_{\overline{\fM_{\varepsilon}}}$ is defined everywhere in ${\overline{\fM_{\varepsilon}}}$, 
by the closed graph theorem, 
it suffices to show that $\bT|_{\overline{\fM_{\varepsilon}}}$ is closed. 
Suppose that 
\[
\|x_n-x\|_{\cH}^2+\|y_n-y\|_{\cH}^2-\|z_n-z\|_{\cH}^2\to 0 \quad(n\to \infty)
\]
and
\[
\|T_1x_n+T_2y_n-T_3z_n-u\|_{\cH}^2\to 0 \quad(n\to \infty).
\]
Then it follows from (\ref{eq:3-1}) that 
$x_n$, $y_n$ and $z_n$ converge to $x$, $y$ and $z$ in $\cH$, respectively. 
Hence we have that 
$u=T_1x+T_2y-T_3z$. This concludes the proof. 
\end{proof}

We will deal with $\bT|_{\overline{\fM_{\varepsilon}}}$ as a Hilbert space operator from 
$\overline{\fM_{\varepsilon}}$ to $\cH$. 
\begin{lem}\label{lem:3-2}
Let $(T_1,T_2,T_3)$ be any triplet in $\fT(\cH)$.  
Then, for any $\varepsilon>0$, 
$\cM(T|_{E_{\varepsilon}^{\perp}\cH})$ 
is contractively embedded into $\cM(\bT|_{\overline{\fM_{\varepsilon}}})$.
\end{lem}

\begin{proof}
Suppose that $u=Tx_{\varepsilon}$ where $x_{\varepsilon}$ is in ${E_{\varepsilon}^{\perp}\cH}$. 
In the proof of Theorem \ref{thm:4-1}, 
we showed that 
\[
\mathbf{z}(\varepsilon)
=\bT^{\sharp}y_{\varepsilon}
=\begin{pmatrix}
T_1^{\ast}y_{\varepsilon}\\
T_2^{\ast}y_{\varepsilon}\\
T_3^{\ast}y_{\varepsilon}
\end{pmatrix}
\] 
satisfies 
$\bT \mathbf{z}(\varepsilon)=Tx_{\varepsilon}=u$
and
\begin{equation*}\label{eq:3-3}
0\leq \la\mathbf{z}(\varepsilon),\mathbf{z}(\varepsilon)\ra_{\cK}=
\|z_1(\varepsilon)\|_{\cH}^2+\|z_2(\varepsilon)\|_{\cH}^2-\|z_3(\varepsilon)\|_{\cH}^2
=\|x_{\varepsilon}\|_{\cH}^2 = \|u\|_{\cM(T|_{E_{\varepsilon}^{\perp}\cH})}^2.
\end{equation*}
Moreover, this $\mathbf{z}(\varepsilon)$ belongs to $\fM_{\varepsilon}$. 
These conclude the proof.
\end{proof}

The next theorem is a generalization of Corollary \ref{cor:3-1}.  
\begin{thm}\label{thm:3-4}\rm 
Let $(T_1,T_2,T_3)$ be any triplet in $\fT(\cH)$.  
Then, for any $\varepsilon>0$, 
\[
\cM(T|_{E_{\varepsilon}^{\perp}\cH})= \cM(\bT|_{\overline{\fM_{\varepsilon}}})
\]
as Hilbert spaces. 
\end{thm}

\begin{proof}
By Lemma \ref{lem:3-2}, 
it suffices to show that 
$\cM(\bT|_{\overline{\fM_{\varepsilon}}})$
is contractively embedded into 
$\cM(T|_{E_{\varepsilon}^{\perp}\cH})$. 
For any $\bT \mathbf{x}$ where $\mathbf{x}=(x_1,x_2,x_3)^t$ 
in $\overline{\fM_{\varepsilon}}$ and any $x$ in $E_{\varepsilon}^{\perp}\cH$, we have that
\begin{align*}
|\la x,\bT \mathbf{x} \ra_{\cH}|^2
&=|\la \bT^{\sharp}x, \mathbf{x} \ra_{\cK}|^2\\
&\leq 
(\|T_1^{\ast} x\|_{\cH}^2+\|T_2^{\ast} x\|_{\cH}^2-\|T_3^{\ast} x\|_{\cH}^2)
(\|x_1\|_{\cH}^2+\|x_2\|_{\cH}^2-\|x_3\|_{\cH}^2)
\end{align*}
because $T^{\sharp}\mathbf{x}=(T_1^{\ast} x,T_2^{\ast} x,T_3^{\ast} x)^t$ 
and $(x_1,x_2,x_3)^t$ belong to $\overline{\fM_{\varepsilon}}$. 
Hence
\[
\sup_{x\in E_{\varepsilon}^{\perp}\cH\setminus\{0\}} \frac{|\la x,\bT \mathbf{x} \ra_{\cH}|^2}
{\|T_1^{\ast} x\|_{\cH}^2+\|T_2^{\ast} x\|_{\cH}^2-\|T_3^{\ast} x\|_{\cH}^2}
=\sup_{x\in E_{\varepsilon}^{\perp}\cH\setminus\{0\}} \frac{|\la x,\bT \mathbf{x} \ra_{\cH}|^2}{\|T x\|_{\cH}^2}
\]
is finite. 
By Theorem \ref{thm:3-3}, $\bT\mathbf{x}$ belongs to $\ran T|_{E_{\varepsilon}^{\perp}\cH}$ and 
\begin{equation*}
\|\bT \mathbf{x}\|_{\cM(T|_{E_{\varepsilon}^{\perp}\cH)}}^2
\leq \|x_1\|_{\cH}^2+\|x_2\|_{\cH}^2-\|x_3\|_{\cH}^2.
\end{equation*}
Thus we have that 
$\|\bT \mathbf{x}\|_{\cM(T|_{E_{\varepsilon}^{\perp}\cH)}}^2
\leq \|\bT \mathbf{x}\|_{\cM ( \bT|_{\overline{\fM_{\varepsilon}}} ) }^2. 
$
This concludes the proof. 
\end{proof}

\end{document}